\documentclass[11pt, lettersize, twoside]{article}
\usepackage{graphicx}
\usepackage{amssymb}
\usepackage{amsmath}
\usepackage{fancyhdr}

\newtheorem{thm}{Theorem}[section]

\newtheorem{prop}[thm]{Proposition}
\newtheorem{defn}[thm]{Definition}
\newtheorem{lem}[thm]{Lemma}

\newcounter{step}
\newcounter{claim}
\newcounter{subclaim}

\newenvironment{proof}{\textbf{Proof}\ }{\hfill $\square$}

\newenvironment{claim}{\vspace{0.75\baselineskip}
\textbf{Claim\ }}{\vspace{0.75\baselineskip}}

\newcommand*{\dateenglish}{\renewcommand*{\today}{%
	\number\day \ifcase\day \or
	st\or nd\or rd\or th\or th\or th\or th\or th\or th\or th\or
	th\or th\or th\or th\or th\or th\or th\or th\or th\or th\or
	st\or nd\or rd\or th\or th\or th\or th\or th\or th\or th\or
	st\fi\space \ifcase\month \or
	January\or February\or March\or April\or May\or June\or
	July\or August\or September\or October\or November\or
	December\fi \space\number\year}}

\newcommand{\gi}{geometrically incompressible}
\newcommand{\gc}{geometrically compressible}

\def\co{\colon\thinspace}

\dateenglish

\begin{document}
\thispagestyle{plain}


\begin{center}
\scalebox{1.5}{\textbf{Structural properties of bounded }}

\scalebox{1.5}{\textbf{one-sided surfaces in link spaces}}
\vspace{0.5\baselineskip}

\textsc{Loretta Bartolini}
\vspace{0.5\baselineskip}

\parbox{.8\textwidth}
{
\textbf{Abstract}
\vspace{.3\baselineskip}

\footnotesize{Various structural properties are developed for non-orientable surfaces in link spaces. The M\"obius band tree is described to represent genus growth of one-sided surfaces in solid tori. The structure of the Tree allows various insights into the change of genus under boundary slope, which are not possible using the existing continued fractions algorithm. A restriction under which \gi{}, boundary compressible one-sided surfaces have a unique boundary incompressible form away from the boundary is established.}
}
\end{center}
\vspace{0.5\baselineskip}

\section{Introduction}

Much successful study in $3$-manifold topology has been focused on identification and manipulation of incompressible surfaces in knot exteriors. For bounded orientable surfaces, boundary compressions generally do not play a role in the global properties of such surfaces. Non-orientable surfaces, however, admit an additional type of boundary compression associated to M\"obius bands in the surface; such compressions give rise to distinct behaviour, which is not necessarily confined to information near the boundary.

Initially, a framework with which to describe \gi{}, one-sided surfaces in a \textit{torus}$\times I$ is developed, encoding information about genus according to boundary slope. This data echoes that from the genus algorithm of Bredon and Wood~\cite{bredon-wood}, however the graphical representation facilitates the identification of trends, particularly with regard to relative nesting.

Having developed such tools for \gi{} surfaces within a \textit{torus}$\times I$, one can seek to extend to more complex manifolds with toroidal boundary. Whilst the structure for behaviour near the boundary is valuable in determining the position of surfaces in link spaces, it is shown that a restriction is required to rule out interaction between boundary components within the interior.

Many thanks to Hyam Rubinstein for helpful discussions throughout the preparation of this paper.

\section{Preliminaries}

Let $M$ be a compact, connected, orientable, irreducible $3$-manifold and let all maps be considered at $PL$.

\begin{defn}
A link space, or link manifold, is a compact $3$-manifold $M$, with non-empty boundary consisting of a collection of tori. Call $M$ a knot space if its boundary is a single torus.
\end{defn}

\begin{defn}
In a link space $M$, a boundary collar is a closed regular neighbourhood of a boundary component; the core of $M$ is the closure of the complement of a complete set of boundary collars.
\end{defn}

\begin{defn}\label{defn:gi}
A surface $K \not= S^2$ embedded in $M$ is \gc{} if there exists an embedded, non-contractible loop on $K$ that bounds an embedded disc in $M$. Call $K$ \gi{} if it is not \gc{}.
\end{defn}

In the event $K$ is orientable, this definition is analogous to $\pi_1$-injectivity. One-sided surfaces, however, admit compressing discs with double points on the boundary. As such, geometric incompressibility of a one-sided surface supports the existence of such singular discs and does not give rise to an algebraic equivalent.

In compact $3$-manifolds with non-empty boundary, the usual notion of boundary compressions generalises to one-sided surfaces:

\begin{defn}
A bounded surface $K \subset M$ with $\partial K \subset \partial M$ is boundary compressible if there exists an embedded bigon $B \subset M$ with $\partial B = \alpha \cup \beta$, such that $\alpha = B \cap K$ is essential, $\beta = B \cap \partial M$ and $\alpha \cap \beta = \partial \alpha = \partial \beta$.
\end{defn}

Since $K$ is one-sided, if $K \cap T_k$ is a single essential loop for some torus $T_k \subseteq \partial M$, a boundary compressing bigon can correspond to the boundary compression of a M\"obius band. In this case, $\beta$ has ends on locally opposite sides of $\partial K$ and the move is referred to as a {\it M\"obius band compression}. If $\beta$ has ends on locally the same side of $\partial K$, this corresponds to a boundary compression in the usual sense of orientable surfaces and is referred to here as an {\it orientable boundary compression}. Note that there are no non-trivial orientable boundary compressions for a \gi{} surface at a single essential boundary loop, since any arc on a torus with both ends on locally the same side of the loop $K \cap T_k$ can be isotoped onto $K$, thus describing a geometric compression.

\section{One-sided surfaces in a \textit{torus}$\times I$}

In order to explore the boundary compressions of \gi{} one-sided surfaces in link spaces, it is necessary to understand the behaviour of such surfaces in a neighbourhood of the boundary. By first characterising the behaviour of \gi{} one-sided surfaces in a solid torus, which can be extended naturally to a \textit{torus}$\times I$, a framework is developed with which to approach the surface in the interior of the space.

It is known that a $(2p,q)$-slope on the boundary of a solid torus bounds a unique incompressible surface in the interior by Rubinstein~\cite{rubin78}. Furthermore, the genus of such a surface can be computed by the continued fractions algorithm of Bredon and Wood~\cite{bredon-wood}. In investigating genus growth over classes of examples, however, such an algorithm is not well-suited to determining trends. As such, one seeks to represent the effect of slope change on genus graphically.

Note that if the core of the solid torus is deleted to obtain a \textit{torus}$\times I$ and the associated puncture of the surface is fixed at the new, inner boundary component, the relative behaviour of the boundary slope at the outer boundary component in the $I$-bundle generalises that of the solid torus. In order to prove this generalisation applies uniquely to surfaces in a \textit{torus}$\times I$, an analogue  is required to Rubinstein's unique determination of bounded \gi{} surfaces in a solid torus by boundary slope~\cite{rubin78}. This is presented here in Proposition~\ref{prop:uniquecollar}, which is given using the structure of the M\"obius band tree.

\subsection{The M\"obius band tree}\label{subsect:tree}

Consider a bounded, \gi{}, one-sided surface embedded in a solid torus, or a \textit{torus}$\times I$. In the solid torus case, choose co-ordinates such that a meridian disc has boundary slope $(0, 1)$; in the latter case, fix one boundary component and choose co-ordinates such that the fixed boundary has slope $(0, 1)$. In both cases, the surface has boundary slope $(2p, q)$ on the remaining boundary component.

Describe a graph $\Gamma$, where a vertex represents an incompressible non-orientable surface of a particular boundary slope, and join two vertices with an edge if the surfaces they bound differ by a single M\"obius band. Algebraically, a vertex corresponds to a slope $(2p,q)$, where $(2p, q)=1$, and an edge joins vertices $(2p, q)$ and $(2p^{\prime}, q^{\prime})$ if the slopes have intersection number $\pm2$. Without loss of generality, represent vertices as co-ordinates $(p,q)$, where $(p,q)=1$, $q$ odd, and edges join vertices $(p,q)$ and $(p^{\prime}, q^{\prime})$ if $pq^{\prime}-p^{\prime}q=\pm1$. Notice that $\Gamma$ is a full subgraph of the Farey graph $\mathcal{F}$ (see Figure~\ref{fig:MBtree}), where the excluded vertices bound disconnected, hence orientable, surfaces in the interior. 

\begin{figure}[h]
   \centering
   \includegraphics[width=4.7in]{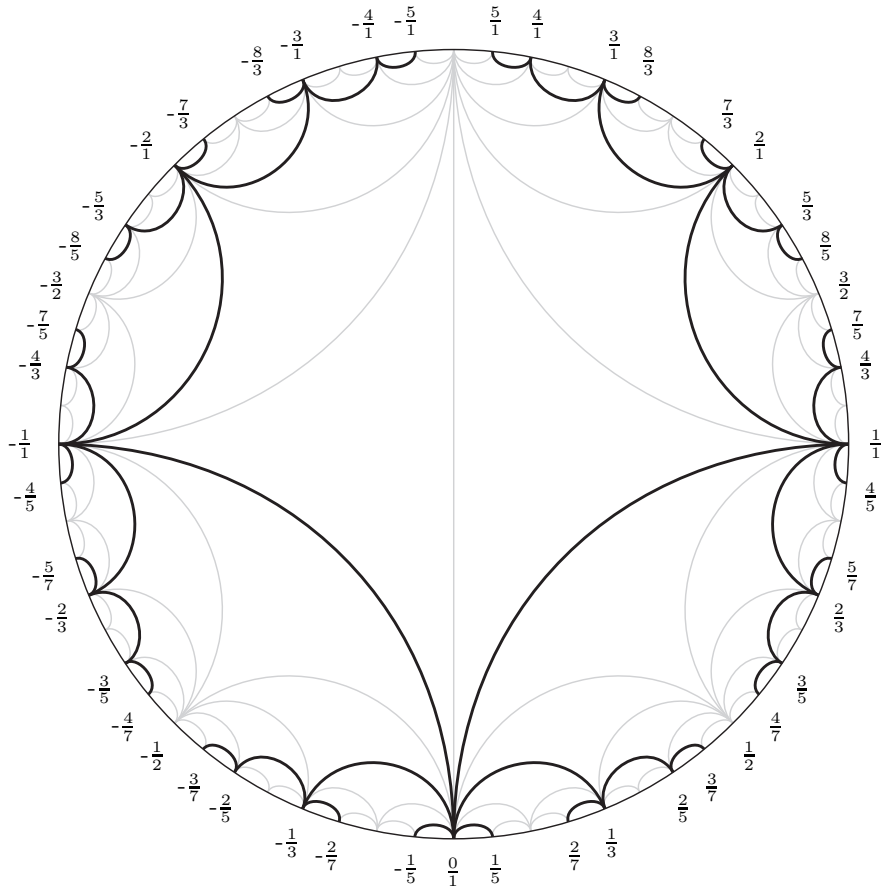}
   \footnotesize
   \put(-173, -2){$\frac{0}{1}$}
   \put(-157, -1){$\frac{1}{5}$}
   \put(-127, 4){$\frac{2}{7}$}
   \put(-113, 9){$\frac{1}{3}$}
   \put(-86, 22){$\frac{2}{5}$}
   \put(-75, 30){$\frac{3}{7}$}
   \put(-61, 42){$\frac{1}{2}$}
   \put(-51, 53){$\frac{4}{7}$}
   \put(-41, 67){$\frac{3}{5}$}
   \put(-27, 94){$\frac{2}{3}$}
   \put(-22, 110){$\frac{5}{7}$}
   \put(-15, 140){$\frac{4}{5}$}
   \put(-13, 160){$\frac{1}{1}$}
   \put(-17, 190){$\frac{4}{3}$}
   \put(-20, 205){$\frac{7}{5}$}
   \put(-25, 220){$\frac{3}{2}$}
   \put(-32, 235){$\frac{8}{5}$}
   \put(-39, 250){$\frac{5}{3}$}
   \put(-57, 273){$\frac{2}{1}$}
   \put(-69, 285){$\frac{7}{3}$}
   \put(-99, 305){$\frac{8}{3}$}
   \put(-112, 311){$\frac{3}{1}$}
   \put(-141, 319){$\frac{4}{1}$}
   \put(-157, 320){$\frac{5}{1}$}
   \put(-192, -1){-$\frac{1}{5}$}
   \put(-222, 4){-$\frac{2}{7}$}
   \put(-237, 9){-$\frac{1}{3}$}
   \put(-262, 22){-$\frac{2}{5}$}
   \put(-274, 30){-$\frac{3}{7}$}
   \put(-288, 42){-$\frac{1}{2}$}
   \put(-299, 53){-$\frac{4}{7}$}
   \put(-308, 67){-$\frac{3}{5}$}
   \put(-324, 94){-$\frac{2}{3}$}
   \put(-329, 110){-$\frac{5}{7}$}
   \put(-336, 140){-$\frac{4}{5}$}
   \put(-338, 160){-$\frac{1}{1}$}
   \put(-334, 190){-$\frac{4}{3}$}
   \put(-331, 205){-$\frac{7}{5}$}
   \put(-326, 220){-$\frac{3}{2}$}
   \put(-319, 235){-$\frac{8}{5}$}
   \put(-310, 250){-$\frac{5}{3}$}
   \put(-292, 273){-$\frac{2}{1}$}
   \put(-278, 285){-$\frac{7}{3}$}
   \put(-250, 305){-$\frac{8}{3}$}
   \put(-235, 311){-$\frac{3}{1}$}
   \put(-208, 319){-$\frac{4}{1}$}
   \put(-192, 320){-$\frac{5}{1}$}
   \caption{The M\"obius band tree $\Gamma$, as it appears in the Farey tessellation of the disc model of $\mathbb{H}^2$}
   \label{fig:MBtree}
\end{figure}

\begin{claim}
$\Gamma$ is a tree.
\end{claim}
\setcounter{subclaim}{1}

$\Gamma$ is connected: Take a vertex $(p, q) \in \Gamma$. Any vertex is in a unique largest ideal triangle in $\mathcal{F}$, where the complementary vertices $(p_1,q_1), (p_1^{\prime},q_1^{\prime}) \in \mathcal{F}$ have $\frac{p_1}{q_1} < \frac{p}{q} < \frac{p_1^{\prime}}{q_1^{\prime}}$ and $|p| > |p_1|, |p_1^{\prime}|$. Since an ideal triangle has precisely two vertices of odd second co-ordinate, one such neighbour \--- say $(p_1,q_1)$ \--- has $q_1$ odd. Therefore, there is an edge in $\Gamma$ joining $(p,q), (p_1,q_1)$. Repeat this process to obtain a neighbour $(p_2, q_2)$ to $(p_1,q_1)$ that has $q_2$ odd and $|p_2| > |p_1|$. Continuing, this process terminates at $(0,1)$, since the magnitude of the first co-ordinate is strictly decreasing. Therefore, there is a path in $\Gamma$ connecting any vertex $(p,q) \in \Gamma$ to $(0,1)$, hence the graph $\Gamma$ is connected. Topologically, this corresponds to the fact that any boundary compressible one-sided surface can be maximally boundary compressed to obtain a boundary incompressible meridian disc or vertical annulus.

$\Gamma$ contains no loops: Suppose there exists an embedded loop $\sigma \subset \Gamma$. Then, considering $\Gamma$ as it appears in the Farey tesselation, $\sigma$ bounds an embedded ideal polygon, which is made up of ideal triangles. An edgemost ideal triangle has all three vertices in $\sigma$. However, any ideal triangle in the Farey tesselation has a vertex $(p,q)$ with $q$ even. Therefore, $\Gamma$ contains no loops.
\vspace{\baselineskip}

Having determined that $\Gamma$ is a tree, observe that the structure inherited from the Farey tessellation illustrates certain natural properties of the tree under the path metric:

\paragraph{Observations}

\begin{enumerate}

\item\label{obs:genus} For vertices $\big\{\frac{p_i}{q_i}\big\}$ on a path rooted at $\frac{0}{1}$, the distance from $\frac{0}{1}$ increases precisely with $|p_i|$;

\item\label{obs:ratio} $\big|\frac{p}{q}\big|$ is closer to $\frac{2}{1}$ than either $\frac{0}{1}$ or $\frac{1}{1}$ if and only if $\frac{3}{2}< \big|\frac{p}{q}\big|$;

\item $\big|\frac{p}{q}\big|$ is closer to $\frac{1}{1}$ than either $\frac{0}{1}$ or $\frac{2}{1}$ if and only if $\frac{1}{2}< \big|\frac{p}{q}\big| <\frac{3}{2}$;

\item $\big|\frac{p}{q}\big|$ is closer to $\frac{0}{1}$ than either $\frac{1}{1}$ or $\frac{2}{1}$ if and only if $\big|\frac{p}{q}\big| < \frac{1}{2}$.
\end{enumerate}

Since the genus of the surface in the solid torus bounded by the curve $(2p, q)$ is the distance in $\Gamma$ of $\frac{p}{q}$ to $\frac{0}{1}$, and any such surface is uniquely determined by its boundary slope, up to isotopy~\cite{rubin78}, these observations can be interpreted as topological information:

By Observation 1, for two surfaces related by boundary compressions, the higher genus surface has longitudinal co-ordinate of larger magnitude.

The relative proximity of a point to the vertices corresponding to the surfaces bounded by curves $(0,1), (2,1), (4,1)$ is governed by proportion of longitudinal to meridional twist in the boundary slope, which in turn dictates the extent to which M\"obius bands nest. A surface that is closest to $(2,1)$ has a single, untwisted M\"obius band at the centremost level of nesting, within which all subsequent bands nest. This leads to a linear relationship between twists in both directions. Meanwhile, a surface closest to $(4, 1)$ has at least two untwisted M\"obius bands at the centremost level of nesting, which determines that longitudinal twisting exceeds meridional in this branch. Furthermore, a surface closest to $(0,1)$ has centremost M\"obius band with at least one meridional twist, giving rise to meridional twisting exceeding longitudinal in this branch. This partition of behaviour, captured in Observations 2\---4 according to the magnitude of the ratio $\frac{p}{q}$, reappears naturally when considering the genus change of non-orientable surfaces under Dehn filling in the following Chapter.

\subsection{Classifying one-sided surfaces in a \textit{torus}$\times I$}

Bounded, \gi{}, one-sided surfaces in a solid torus are uniquely determined by boundary slope, up to isotopy, by Rubinstein~\cite{rubin78}. Therefore, vertices and paths in the M\"obius band tree correspond precisely to surfaces, and compressions thereof, up to isotopy. In order to show that the M\"obius band tree similarly precisely captures bounded, \gi{} surfaces in a \textit{torus}$\times I$, it is necessary to show that any such surface is uniquely determined by its boundary slopes:

\begin{prop}\label{prop:uniquecollar}
Within a \textit{torus}$\times I$, any pair of inner and outer boundary slopes describes a unique geometrically incompressible one-sided surface, up to isotopy.
\end{prop}

\begin{proof}
Consider a bounded, \gi{}, non-orientable surface $K_T$ in a \textit{torus}$\times I$ $N$, with boundary curves $s_T \subset \partial_{-}N, t_T \subset \partial_{+}N$. If the slopes of $s_T, t_T$ differ, the surface has non-trivial non-orientable genus $g$.

Parametrise $N = T \times \{t\}$ by $t \in [0,1]$, where $T \times \{0\} = \partial_{-}N$ and $T \times \{1\} = \partial_{+}N$. In $N^{\circ}$, all but finitely many tori $T \times \{t_i\}$, $1\leq i \leq g-1$, intersect $K_T$ transversely. Arrange so there is precisely one connected tangency arc at any critical level. At the non-critical levels, any inessential loops of intersection are trivial on $K_T$ by geometric incompressibility and can be removed by trivial surgery. The remaining intersection $K_T \cap (T \times \{t\})$ for $t \in (t_{i-1}, t_{i})$ is a single essential loop with slope $(2p_i, q_i)$. Let $(2p_g, q_g)$ be the slope of $K \cap (T \times \{1\})$. At each critical point $\{t_i\}$, the level torus is tangent to the core of a single M\"obius band. Thus, $N$ can be partitioned into $g$ concentric \textit{torus}$\times I$ regions, where a region $T_i \subset N$ is a regular neighbourhood of the critical level at $t_i$. Each region thereofre contains a single, once-punctured M\"obius band with boundary slope $(2p_i, q_i)$ on $\partial_{-}T_i$ and $(2p_{i+1}, q_{i+1})$ on $\partial_{+}T_i$.

Each pair of slopes $(2p_{i}, q_{i}), (2p_{i+1}, q_{i+1})$ cobound a single once-punctured M\"obius band in the region $T_i$, so $(2p_{i}, q_{i}), (2p_{i+1}, q_{i+1})$ have intersection number $\pm 2$. Therefore, in the M\"obius band tree, the sequence of slopes at the boundaries of the regions corresponds to a path from $(2p_1, q_1) = s_1$ to $(2p_{g}, q_{g}) = t_1$. Such a path is minimal, since any backtracking along side branches involves at least one pair of regions $T_{i}, T_{i+1}$ where $(2p_{i}, q_{i}) = (2p_{i+2}, q_{i+2})$, which have compressing discs running between all M\"obius bands centremost to $\partial_{+}T_i = \partial_{-}T_{i+1}$. Since minimal paths in a tree are unique, the sequence of slopes of intersection curves is unique.

\begin{lem}\label{lem:MBtorusxI}
If $N_0$ is a \textit{torus}$\times I$, then a pair of slopes $(2p, q), (2p^{\prime}, q^{\prime})$, on $\partial_{+}N, \partial_{-}N$ respectively,  where $|pq^{\prime}-p^{\prime}q|=1$, bound a unique, embedded, \gi{}, once-punctured M\"obius band, up to isotopy.
\end{lem}

\begin{proof}
Suppose $K_1, K_2 \subset N_0$ are embedded, bounded, \gi{}, non-orientable surfaces with boundary slopes $(2p, q), (2p^{\prime}, q^{\prime})$ on $\partial_{-}N_0$, $\partial_{+}N_0$, respectively, such that $|pq^{\prime}-p^{\prime}q|=1$. Isotope the surfaces near $\partial N_0$ so that $\partial K_1 \cap \partial K_2 = \emptyset$. Since both surfaces are \gi{}, any contractible loops of intersection in $K_1 \cap K_2$ bound trivial discs on both surfaces, hence can be removed by trivial surgery.

Given $K_1, K_2$ each boundary compress to a vertical annulus over $(2p, q)$ via a single non-orientable boundary compression, they both have non-orientable genus one. Since the surfaces correspond to the same $\mathbb{Z}_2$-homology class, each complement $K_i \setminus K_j$, for $i\not=j$, is orientable. Therefore, the intersection $K_1 \cap K_2$ includes a non-contractible loop $\sigma$ that forms the core of the M\"obius band in $K_1$ such that $K_1 \setminus \sigma$ is orientable. By geometric incompressibility, the loop $\sigma$ is likewise non-contractible on $K_2$. Since $\sigma$ is not parallel to either component of $\partial K_1$ and $\partial K_1 \cap \partial K_2 = \emptyset$, nor is the loop $\sigma$ boundary parallel on $K_2$. However, $K_2$ has genus one, hence the only disjoint parallel non-contractible loops are parallel into the boundary. Therefore, $\sigma$ is also a non-boundary parallel, non-contractible loop on $K_2$, which forms the core of the M\"obius band in $K_2$ such that $K_2 \setminus \sigma$ is orientable.

Let $A \subset N_0$ be an embedded vertical annulus over the unique slope that has intersection $\pm1$ with both $(2p, q), (2p^{\prime}, q^{\prime})$. Therefore, $A$ can be isotoped to intersect each loop in $\{\partial_{+}K_i, \partial_{-}K_i\}$ in a single point. The intersections $K_i \cap A$ consist of a collection of loops and arcs. By geometric incompressibility, every loop of intersection is trivial and can be removed by trivial surgery. By construction, each surface $K_i$ has a single arc $\lambda_i \subset (K_i \cap A)$ such that $\partial \lambda_i$ has one endpoint on each of $\partial_{-}K_i, \partial_{+}K_i$. The intersection of the two arcs $\lambda_1 \cap \lambda_2$ contains the single point $\sigma \cap A$.

Any other arc in $(K_i \cap A) \setminus \lambda_i$ has endpoints on precisely one boundary component. Without loss of generality, suppose $\alpha_1 \subset (K_1 \cap A) \setminus \lambda_1$ is an edgemost arc with endpoints on $\partial_{+}K_1$, hence cobounds an embedded bigon $b_1 \subset A$ with $\partial_{+}N_0$. Since $\alpha_1$ is disjoint from $\lambda_1$, hence $\sigma$, the bigon $b_1$ represents a trivial boundary compression. Working from edgemost arcs inwards, all arcs in $(K_1 \cap A) \setminus \lambda_1$ with endpoints on $\partial_{+}N_0$ can can be removed by trivial surgery. Similarly, arcs in $(K_1 \cap A) \setminus \lambda_1$ with endpoints on $\partial_{-}N_0$ can be removed, as can arcs in $(K_2 \cap A) \setminus \lambda_2$ with both endpoints at either boundary component.

The remaining intersection $A \cap (K_1 \cup K_2)$ consists entirely of $\lambda_1, \lambda_2$. Without loss of generality, consider a point $x_1 \subset (\lambda_1 \cap \lambda_2)$ that is edgemost with respect to $\partial_{+}A$. This point lies in a loop $\sigma_1 \subset (K_1 \cap K_2)$. If $\sigma_1$ is contractible in either surface, it is likewise contractible in the other and can be removed by geometric incompressibility. Suppose $\sigma_1$ is non-contractible on both $K_i$. Since $K_1, K_2$ have non-orientable genus one and are \gi{}, there is a unique, non-boundary-parallel, non-contractible loop in each surface, up to isotopy. As any loop isotopic to $\sigma$ crosses $\sigma$, no such loop occurs in $K_1 \cap K_2$. Therefore, $\sigma_1$ is boundary parallel in both surfaces.

Since $\sigma_1$ is boundary parallel, take annuli $A_i \subset K_i$ cobounded by $\sigma_1, \partial_{+}A$ such that $A_1 \cap A_2 = \sigma_1$. The union $A_1 \cup A_2$ is boundary parallel, hence bounds a solid torus $N_T$ with the region of $\partial_{+}N_0$ that runs between the parallel boundaries $\partial_{+}K_i$. If either $K_i \cap N_T \not= \emptyset$, the subsurface $K_i \cap N_T$ has boundary only on $A_j$, $i\not=j$. The intersection $K_i \cap A_j$ consists of a collection of loops, any of which that are contractible on either surface can be removed by geometric incompressibility. There are an even number of non-contractible loops in $K_i \cap A_j$, which are parallel on $A_j$ and boundary parallel, hence mutually parallel, on $K_i$. An adjacent pair $\sigma_2, \sigma_3 \subset (K_i \cap A_j)$, where the intervening region $\bar{A}_j \subset A_j$ is disjoint from $K_i$, cobounds a boundary parallel annulus in $K_i$. Surger $K_i$ along the annulus $\bar{A}_j$ to obtain an isotopic surface, less the intersections $\sigma_2, \sigma_3$ with $A_j$. Iterate this process to remove all intersections in $K_i \cap A_j$, $i\not=j$. Hence $K_i \cap N_T = A_i$ for $i=1, 2$.

Surger $K_1$ along the annulus $A_2$ to yield a surface, isotopic in $N_0$, less the intersection $\sigma_1$ with $K_2$. The effect on $A$ is to isotope $\lambda_1$ across $\lambda_2$ near $\partial_{+}A$ to remove the point $x_1$. Iterate this process, working from points in $\lambda_1 \cap \lambda_2$ edgemost to $\partial_{+}A$, inwards towards $\sigma$, to remove all boundary parallel intersections in $K_1 \cap K_2$ that are parallel to $\partial_{+}K_i$. Similarly, all loops in $K_1 \cap K_2$ that are parallel to the opposite boundary component $\partial_{-}K_i$ can be removed. Therefore, $K_1 \cap K_2$ is the single non-contractible loop $\sigma$.

Consider the subdivision of $A$ by the intersection $A \cap \{K_i\} = \{\lambda_i\}$. Two components are three-sided discs $d_1, d_2$, where $\partial d_1$ has one arc on each of $K_1, K_2, \partial_{+}N_0$ and  $\partial d_2$ similarly has one arc on each of $K_1, K_2, \partial_{-}N_0$, such that $\partial d_1 \cap \partial_{+}N_0, \partial d_2 \cap \partial_{-}N_0$ are arcs traversing the region between the parallel boundary curves $\partial_{\pm}K_i$. The discs intersect precisely at $\sigma$, at which they are on locally opposite sides of each $K_i$. The remaining region is a six-sided disc, with a single edge on each of $\partial_{+}N_0, \partial_{-}N_0$ and two disjoint edges on each of $\lambda_1, \lambda_2$.

Isotope the surfaces $K_1, K_2$ in a neighbourhood of the boundary so that $\partial K_1, \partial K_2$ coincide at both boundary components. The effect on $d_1, d_2$ is to shrink the arcs on $\partial_{+}N_0, \partial_{-}N_0$, respectively, to points. The resulting discs therefore have two boundary arcs, one on each of $K_1, K_2$.  

In a closed regular neighbourhood $N(\sigma)$ of $\sigma$, the restrictions $K_i \cap N(\sigma)$ are single M\"obius bands with disjoint boundary, hence are isotopic by Rubinstein~\cite{rubin78}. Isotope the surfaces together in $N(\sigma)$ to yield once-punctured annuli $\bar{K}_i = \overline{K_i \setminus N(\sigma)}$ in the complement, which meet precisely in their boundary slopes.

The union $S = \bar{K}_1 \cup \bar{K}_2$ is a closed, embedded, orientable surface of genus two, where the intersection with each of $\partial_{+}K_2, \sigma, \partial_{-}K_2$ is a non-contractible loop on $S$. Since any incompressible surface in $N_0$ is parallel to a boundary torus, the genus two surface $S$ can be compressed to yield a torus $\bar{S}$, which intersects $\partial N(\sigma)$ and one of $\partial_{+}K_2, \partial_{-}K_2$ in a non-contractible loop. Since no such surface is boundary parallel in $N_0 \setminus N(\sigma)$, the torus $\bar{S}$ is compressible, after which it bounds a $3$-cell, by irreducibility. Therefore, $S$ is fully compressible.

If $\bar{d}_i = \overline{d_i \setminus N(\sigma)}$ for $i=1,2$, then each of $\bar{d}_i$ is an embedded disc with boundary consisting of one arc on each of $\bar{K}_1, \bar{K}_2$, such that $\partial \bar{d}_i$ intersects $\partial N(\alpha)$ and one of $\partial_{\pm}K_2$ precisely once. By construction, such a loop is non-contractible on $S$, so each disc represents a compression of $S$. Since the original discs $d_1, d_2$ are on locally opposite sides of $K_i$, for $i=1,2$, the discs $\bar{d}_1, \bar{d}_2$ are on opposite sides of $S$. Hence, $\bar{d}_1, \bar{d}_2$ are not equivalent and represent a complete set of compressions for $S$.

Homotope the surface $\bar{K}_1$ across the discs $\bar{d}_1, \bar{d}_2$. Since this homotopy effectively fully compresses the surface $S$, further isotopy can ensure that $\bar{K}_1$ fully coincides with $\bar{K}_2$ in $N_0 \setminus N(\sigma)$. Since $\bar{K}_1$ is homotopic to $\bar{K}_2$ in $N_0 \setminus N(\sigma)$, which is Haken, $\bar{K}_1$ is therefore isotopic to $\bar{K}_2$ in the same manifold, by a result of Waldhausen~\cite{wald_suff_large}. Combined with the fact that the restrictions are isotopic in the neighbourhood $N(\sigma)$, the original M\"obius bands $K_1$ and $K_2$ are isotopic in $N_0$. Hence a \gi{}, non-orientable surface in a \textit{torus}$\times I$ is uniquely determined by a pair of boundary slopes with intersection number $\pm2$.\end{proof}

Having thus proved Lemma~\ref{lem:MBtorusxI}, each adjacent pair of edges in the M\"obius band tree uniquely determines a punctured M\"obius band, up to isotopy. Since the minimal sequence of boundary slopes in the M\"obius band tree connecting any pair of slopes is unique, any pair of inner and outer boundary slopes on $\partial_{-}N, \partial_{+}N$ describes a unique geometrically incompressible one-sided surface in $N$, up to isotopy.
\end{proof}
\vspace{0.1\baselineskip}

If one boundary component of the \textit{torus}$\times I$ is filled, this Proposition provides an alternative proof to the solid torus case by Rubinstein~\cite{rubin78}.
\vspace{\baselineskip}

Whilst the M\"obius band tree can be applied equally to study the behaviour of surfaces in a solid torus and those in a \textit{torus}$\times I$, there is an important distinction: a \gi{} non-orientable surface in a \textit{torus}$\times I$ can be compressed toward either of the inner or outer boundary components. Choosing either boundary component as fixed, the M\"obius band tree can be used to track the changing slopes at the opposite boundary. Although the tree structure is independent from choice of co-ordinates, in order that the comments on genus apply, it is necessary to choose co-ordinates such that the fixed boundary has slope $(0,1)$.
\vspace{\baselineskip}

Remark that the uniqueness of minimal paths in the M\"obius band tree does not imply that boundary compressions themselves are unique: compressing different, non-parallel M\"obius bands along inequivalent bigons gives rise to the same boundary slope.

Consider the example of the $(10, 3)$ surface in a solid torus $M_T$. Describe a given M\"obius band by the integer pair $(a, b)$: let $a$ be the intersection number of the parallel boundary arcs of the band on $\partial M_T$ with a $(0,1)$ curve, that intersects their interior only; and, let $b$ similarly be the intersection number with the curve $(1,0)$. In the example, the $(10, 3)$ surface has two $(2, 0)$ M\"obius bands and one $(6, 2)$ band. Compressing the nested $(6, 2)$ M\"obius band results in a $(4, 1)$ boundary curve, as does compressing the un-nested $(2, 0)$ band. Refer to Figures~\ref{fig:(4,1)} and~\ref{fig:(4,1)_1}.

\begin{figure}[h]
   \centering
   \includegraphics[width=2.5in]{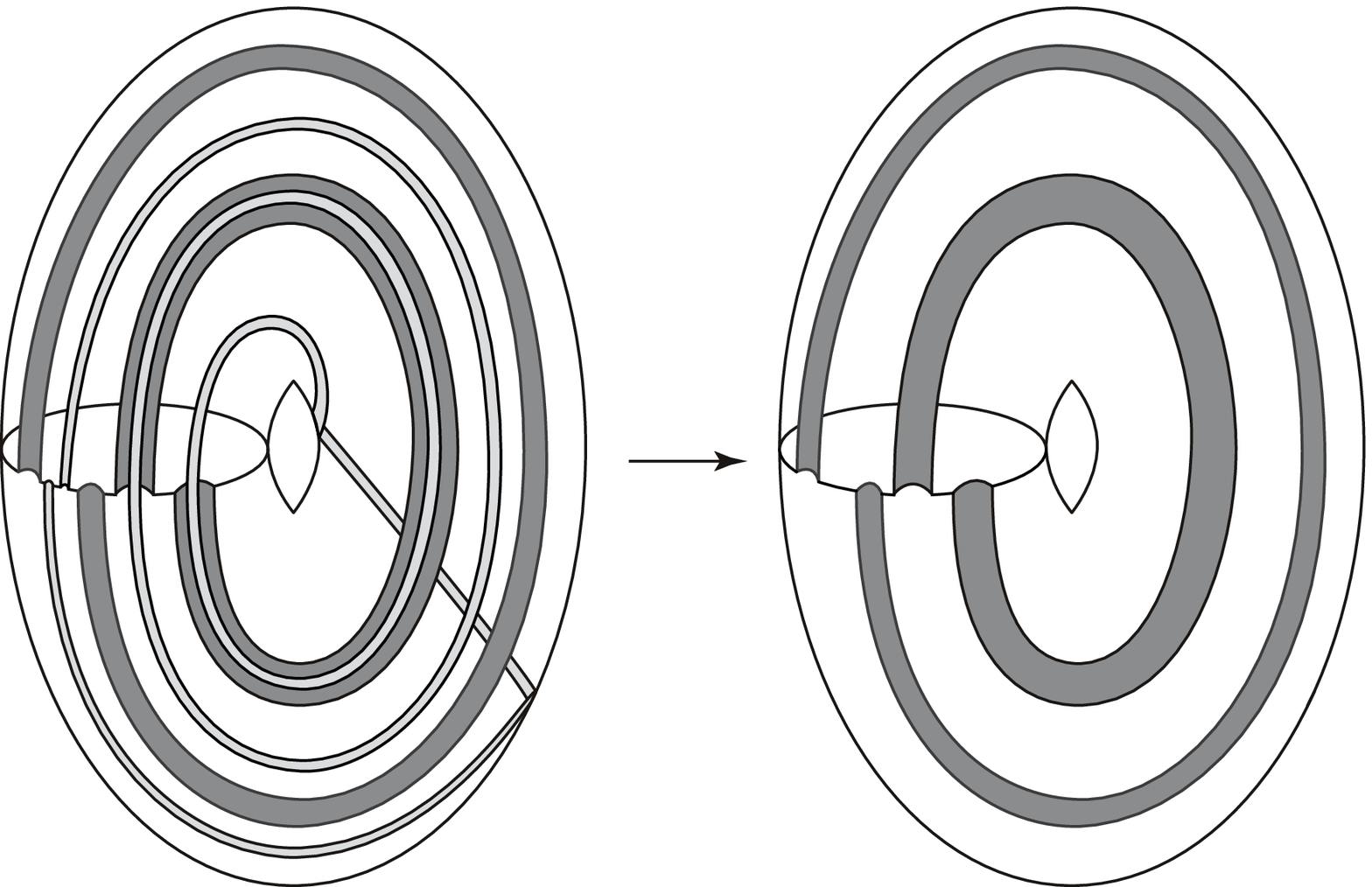}
   \put(-110, 10){\small compress}
   \caption{Boundary compressing the $(6, 2)$ M\"obius band of the $(10, 3)$ surface to get the $(4, 1)$ curve.}
   \label{fig:(4,1)}
   \vspace{\baselineskip}
   \includegraphics[width=5in]{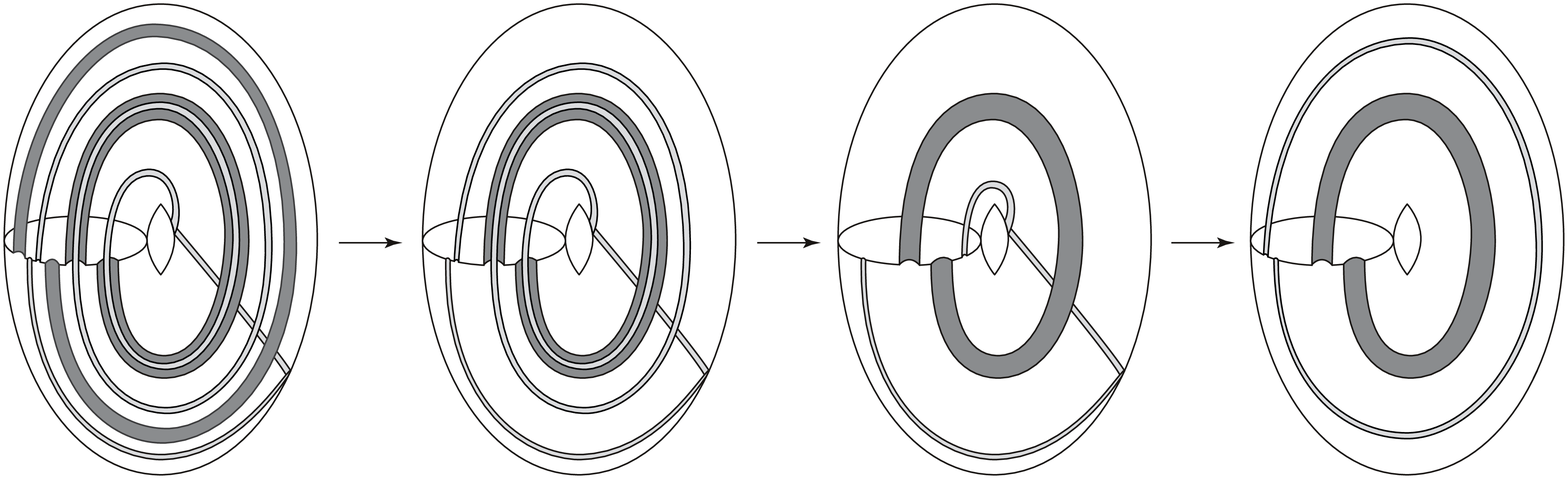}
   \put(-296, 4){\small compress}
   \put(-190, 4){\small slide}
   \put(-94, 4){\small slide}
   \caption{Boundary compressing the $(2, 0)$ M\"obius band of the $(10, 3)$ surface to get the $(4, 1)$ curve.}
   \label{fig:(4,1)_1}
\end{figure}

\subsection{Adding M\"obius bands to bounded surfaces in link spaces}

Having determined the tree structure that relates the change in boundary compressible M\"obius bands in a solid torus and \textit{torus}$\times I$, consider the implications for adding boundary compressible M\"obius bands to surfaces in link spaces:

\begin{lem}\label{lem:boundary_curves}
Given a bounded, \gi{}, boundary incompressible surface $K_0$ in a link space $M$, there is at most one \gi{}, boundary compressible surface $K$, with a given boundary slope $\sigma_k$ on $T_k$, that can be obtained by attaching M\"obius bands to $K_0$.
\end{lem}

\begin{proof}
Suppose, by way of contradiction, that there exist two bounded, incompressible surfaces $K_1, K_2 \subset M$, each with boundary slope $\sigma_k$ on boundary torus $T_k$, that are obtained by attaching M\"obius bands to the same boundary incompressible surface $K_0 \subset M$. For each boundary torus $T_k$, isotope the surfaces $K_1, K_2$ in a collar $N_k$ such that their boundary curves are both $s_k \subset T_k$. Further isotope the surfaces in the core $M_0 = M \setminus \{N_k\}$, so that they coincide with $K_0$ in $M_0$.

Consider $K_1, K_2$ in the set of boundary collars $\{N_k\}$. On the inner boundaries $\{\partial_{-}N_k\} = \{T_k\}$, the surfaces have the same boundary curves $\{s_k\}$. Having isotoped the surfaces to coincide in the core, their set of intersections with the outer boundary $\{\partial_{+} N_k\} = \partial M_0$ again coincides in the set of curves $\{t_k\} = \partial K_0 \cap \partial M_0$.

By Proposition~\ref{prop:uniquecollar}, the component-wise matching boundary slopes determine that the restrictions $K_1 \cap N_k$, $K_2 \cap N_k$ are isotopic in each collar, respectively. Since $K_1, K_2$ therefore coincide in both the core and each collar, they are isotopic in $M$.
\end{proof}

\section{Canonical structure for one-sided surfaces in link spaces}

Having determined a means for comparing the arrangement of \gi{} one-sided surfaces in a \textit{torus}$\times I$, it is now possible to progress the discussion to such surfaces in general link spaces. Given a \gi{} one-sided surface in a link space, it is useful to know how the surface can look if all the boundary compressible M\"obius bands are arranged into collars of the boundary. In particular, is there a unique surface away from the boundary, or can the rearrangement of M\"obius bands result in non-isotopic surfaces in the core?

It is determined that under a given restriction, there is a unique boundary incompressible position in the core, which is determined by any complete set of boundary compressing bigons. However, if there is a particular type of embedded disc between boundary components in the manifold, then M\"obius bands can slide between the components, resulting in non-isotopic restrictions in the core.

\begin{defn}\label{defn:polygonal}
A disc $d \subset M$ for a bounded surface $K$ in a link space $M$ is polygonal if $d^{\circ}$ is embedded and $\partial d$ is partitioned into $2n$ arcs, $\lambda_1, \lambda_2, \ldots, \lambda_{2n}$ for $n \geq 2$, where $\lambda_i \cap \partial K = \partial \lambda_i$ for all $i$, and $\lambda_{2k-1} \subset K$ and $\lambda_{2k} \subset \partial M$ for $k=1, \ldots, n$. Each arc $\lambda_{2k-1}$ is essential in $K$ and each arc $\lambda_{2k}$ is embedded, disjoint from all other $\lambda_{2s}$, $s \not= k$ and not homotopic into $\partial K$ along $\partial M$ rel endpoints.
\end{defn}

Note that if $K$ is non-separating, there is the usual potential for isolated singularities on a non-orientable surface, arising from the possibility that $\partial \lambda_{2k-1} \cap \partial \lambda_{2s-1} \not= \emptyset$ for $s \not= k$, where $d$ is on locally opposite sides of $K$. If $\{\lambda_{2k-1}\}$ is a disjoint set, then the polygonal disc $d$ is embedded.

\begin{defn}
A bounded surface in a link space $M$ is said to be `at' a collection of boundary tori $\{T_k\}$ if it intersects each element in the set in at least one arc, or essential loop, and is otherwise disjoint from $\partial M$.
\end{defn}

It is useful to describe both polygonal discs and bounded non-planar surfaces as being {\it at} a particular collection of tori; the former intersecting the tori in a collection of boundary arcs, the latter, a collection of essential loops in the surface boundary.

\begin{defn}
An embedded polygonal disc $d$ for a bounded surface $K$ in a link space $M$ is said to be reducible if there exists an embedded bigon $b \subset M$, with $\partial b = \alpha \cup \beta$ and $\alpha \cap \beta = \partial \alpha = \partial \beta$, where $\alpha = b \cap d$ and $\beta = b \cap K$, such that $\beta \cap d = \partial \beta$ and the endpoints of $\alpha$ lie on distinct boundary arcs $\lambda_{2k-1}, \lambda_{2s-1}$, $k \not= s$. Call $d$ irreducible if it is not reducible.
\end{defn}

Compressing a reducible polygonal disc $d$ along the bigon $b$, which is otherwise disjoint from $d$, results in two embedded polygonal discs of strictly lower complexity. If either such disc is reducible, this process may continue until a set of minimal embedded polygonal discs is obtained, which may include boundary compressing bigons. Indeed, by band-summing sets of embedded bigons together, fully reducible polygonal discs can be produced for any boundary compressible surface. However, reductions of polygonal discs for boundary incompressible surfaces cannot produce non-trivial bigons, hence quadrilaterals are the lowest complexity irreducible polygonal disc possible for such a surface.

Note that the presence of irreducible quadrilateral discs at a single boundary component is typical of a fibred knot complement, however, generically these will have isolated singularities due to the twist from the monodromy.

\begin{prop}\label{prop:boundary_incomp}
Any bounded, \gi{} one-sided surface embedded in a link manifold that is at the collection $\{T_k\} \subseteq \partial M$ has a boundary incompressible restriction outside of a neighbourhood of $\{T_k\}$, which is unique up to isotopy if the manifold contains no non-trivial, embedded, irreducible, polygonal discs for the boundary incompressible surface at $\{T_k\}$.
\end{prop}

\begin{proof}
Take $K$, a bounded \gi{} surface embedded in a manifold $M$ such that $K$ has essential intersection with the collection $\{T_{k}\} \subseteq \partial M$. Let $\{N_k\} \subset M$ be a set of closed regular neighbourhoods of these boundary tori, with $N_k$ a neighbourhood of $T_{k}$ such that $N_k \cap K$ is a regular neighbourhood of $T_k \cap K$ for all $k$. Let $M_0 = \overline{M \setminus \{N_k\}}$ be the core of $M$ and call $M_{k} = \overline{M \setminus N_k}$ the {\it associated core} with respect to a single boundary collar $N_k$. Let $\partial_{-}N_k = T_k$, such that $\{\partial_{+}N_k\} = \partial M_0$. Fix the surface $K$ at the link space boundary $\{T_k\}$ throughout. 

Consider boundary compressions of $K$ in $M_0$ toward $\{\partial_{+}N_k\}$. By geometric incompressibility, all such compressions are non-orientable, hence describe isotopies of $K$ that move M\"obius bands into the boundary collars.

Suppose $K$ can be boundary compressed across $\partial M_0$ in distinct ways, such that non-isotopic boundary incompressible restrictions can be obtained in the core. Rather than moving the surface $K$ to achieve either set of boundary compressions, capture this behaviour by different sets of boundary collars:

Let $\{N_k^1\}$ be a disjoint set of boundary collars for $\partial M$ with $\{N_k\} \subset \{N_k^1\}$, where $\{N_k^1\}$ is obtained by expanding $\{N_k\}$ such that the restriction $K_1 = K \cap \overline{M \setminus \{N_k^1\}}$ is boundary incompressible: 

If $\mathcal{B}_1^{1}$ is a maximal disjoint set of embedded boundary compressing bigons cobounded by $K$ and $\{\partial_{+}N_k\}$, expand $N_1$ to include $\mathcal{B}_1^{1}$. If $K$ is boundary compressible in the complement of the expanded collars, there is a new maximal disjoint set of embedded bigons $\mathcal{B}_2^{1}$, in which the previous bigons nest; expand the collars again to include this set. Continue this process of expansion along $\{\mathcal{B}_i^{1}\}$, where $1 \leq i \leq l$ denotes level of nesting, to obtain maximally expanded collars $\{N_k^1\}$, in the complement of which $K$ is boundary incompressible.

Suppose there exists an alternative disjoint set of boundary collars $\{N_k^n\}$ with $\{N_k\} \subset \{N_k^n\}$, obtained by expanding $\{N_k\}$ along an alternative nested collection of maximal disjoint sets of boundary compressing bigons $\{\mathcal{B}_j^n\}$, where $1 \leq j \leq m$, such that the restriction $K_n = K \cap \overline{M \setminus \{N_k^n\}}$ is also boundary incompressible. Finally, assume that after identifying the cores $\overline{M \setminus \{N_k^1\}}$ and $\overline{M \setminus \{N_k^n\}}$ by a homeomorphism $\gamma \co \overline{M \setminus \{N_k^1\}} \rightarrow \overline{M \setminus \{N_k^n\}}$, that $\gamma(K_1)$ is not isotopic to $K_n$ in $\overline{M \setminus \{N_k^n\}}$.

Since each surface is boundary incompressible in the core, the boundary compressible M\"obius bands of $K$ are variously arranged in the expanded sets of boundary collars $\{N_k^1\}, \{N_k^n\}$. Each collar is a \textit{torus}$\times I$ and $K$ is fixed at $\{T_k\}$ throughout, therefore, examine such subsurfaces via the slopes at $\{\partial_{+}N_k^1\}, \{\partial_{+}N_k^n\}$ using the M\"obius band tree:

Let $\bar{K}^1 = K \cap \{N_k^1\}$ and $\bar{K}^n = K \cap \{N_k^n\}$. If the boundary slopes $K \cap \partial_{+}N_k^1$ and $K \cap \partial_{+}N_k^n$ coincide at every boundary component, then the restrictions $\bar{K}^1, \bar{K}^n$ are component-wise isotopic in the respective collars. Since the surfaces $K_1 \cup \bar{K}^1$ and $K_n \cup \bar{K}^n$ are both isotopic to $K$ in the link space $M$, after the above homeomorphism, the remaining restrictions $\gamma(K_1), K_n$ are isotopic in the remaining piece: $\overline{M \setminus \{N_k^n\}}$. Therefore, if any two maximal sets of boundary compressions result in surfaces with the same sets of boundary curves at all boundary components, then the boundary incompressible surfaces are isotopic in the core.

Suppose one pair of boundary slopes for $K \cap \partial_{+}N_1^1, K \cap \partial_{+}N_1^n$, differs. Choose co-ordinates for $T_1$ to apply to both collars, such that $(0,1) = K \cap T_1$, hence $(2p_1, q_1) = K_1 \cap \partial_{+}N_1^n$ and $(2p_n, q_n) = K_n \cap \partial_{+}N_1^n$. The slopes correspond to distinct vertices $(p_1, q_1)$ and $(p_n, q_n)$ in the M\"obius band tree $\Gamma$, which are joined by a shortest path $\lambda \subset \Gamma$. Enumerate vertices in $\lambda$, starting at $(p_1, q_1)$ and ending at $(p_n, q_n)$. Whilst each of $K_1, K_n$ are obtained by boundary compressing $K$ out of the core, the vertex $(0,1)$ corresponding to $K$ may not be included in $\lambda$ if there are a number of common boundary compressions.

\begin{claim}
The slopes of the boundary curves $K \cap \partial_{+}N_1^1$ and $K \cap \partial_{+}N_1^n$ differ only if there exist embedded, irreducible polygonal discs in the cores $(M \setminus \{N_k^1\}) \cup \partial_{+}N_1^1, (M \setminus \{N_k^n\}) \cup \partial_{+}N_1^n$ for $K_1, K_n$, respectively.
\end{claim}

Consider $N_1^1$ with reference to $N_1^n$. Since the slopes $K \cap \partial_{+}N_1^1$ and $K \cap \partial_{+}N_1^n$ differ, each collar contains a topologically different subsurface of $K$: after identifying the cores via the homeomorphism $\gamma$ described above, the restrictions $\gamma(K_1^{\prime}), K_n^{\prime}$ are not isotopic in $N_1^n$. Note that either or both of the genus and nesting may differ between the two surfaces.

If one collar is a proper subset of the other, say $N_1^1 \subset N_1^n$, then the restriction $\bar{K}_1^1 = K \cap N_1^1$ is a subsurface of $\bar{K}_1^n = K \cap N_1^n$. Since the boundary slopes on $N_1^1, N_1^n$ differ, the complement $\bar{K}_1^n \setminus \bar{K}_1^1$ is non-trivial in the subcollar $N_1^n \setminus N_1^1$. Therefore, there exists a boundary compressing bigon cobounded by $\bar{K}_1^n \setminus \bar{K}_1^1$ and $\partial_{+}N_1^1$ embedded in $N_1^n \setminus N_1^1$, contradicting the boundary incompressibility of $K_1$ in $\overline{M \setminus \{N_k^1\}}$. Hence, neither collar is a proper subset of the other.

Since $N_1^n \setminus N_1^1 \not=\emptyset$, there exists an element of $\{\mathcal{B}_j^n\}$ that intersects $N_1^n \setminus N_1^1$ non-trivially. Let $B \in \{\mathcal{B}_j^n\}$ be an innermost such bigon. If $B \setminus N_1^1$ is a bigon in $\overline{M \setminus \{N_k^1\}}$, then this contradicts the geometric incompressibility of $K_1$ in the core. Therefore, $B$ is a bigon that non-trivially intersects $\{N_k^1\}$.

As the boundary slopes at all other collars agree, $N_k^1 = N_k^n$ for $k \not=1$, therefore since each of $\{N_k^1\}, \{N_k^n\}$ is a disjoint set, the fact $B \subset N_1^n$ determines that  $B \cap N_k^1 = \emptyset$ for $k \not=1$. Therefore, $B$ non-trivially intersects $N_1^1$. Hence, the restriction $\bar{B} = B \cap \overline{M \setminus \{N_k^1\}}$ is an embedded irreducible polygonal disc with boundary arcs alternating between $K_1$ and $\partial_{+}N_1^1$.

Therefore, under the restriction that no embedded, irreducible polygonal discs exist for boundary incompressible surfaces in the core, there are no boundary incompressible restrictions with different boundary slopes at a single boundary component.

If the boundary slopes at the expanded collars differ at boundary components $1, 2, \ldots, s$, then a similar argument can be applied. However, the non-empty intersection of elements in $\{N_k^1\}, \{N_k^n\}$ may occur between any $N_k^1, N_l^n$, with $k, l \in 1, 2, \ldots, s$. If $N_1^1, N_2^n$ are such an intersecting pair, an innermost embedded bigon $B^{\prime} \subset N_2^n$ with $B^{\prime} \cap (N_2^n \setminus N_1^1)$ non-trivial can be chosen as above. However, in the final step, the restriction $\bar{B}^{\prime} \subset \overline{M \setminus \{N_k^1\}}$ intersects $N_1^1, N_2^n$ as above, however may also intersect any other collar $N_k^1$ for $k \in 1, 2, \ldots, s$. Hence, $\bar{B}^{\prime}$ is an embedded irreducible polygonal disc with boundary arcs alternating between $K_1$ and any of $\{\partial_{+}N_k^1\}$ for $k \in 1, 2, \ldots, s$.

Therefore, under the restriction that no embedded, irreducible polygonal discs exist for boundary incompressible surfaces in the core, there are no boundary incompressible restrictions with differing boundary slopes any collection of boundary component. Hence, in thus restricted link spaces, any bounded, \gi{} one-sided surface embedded at the collection $\{T_k\} \subseteq \partial M$ has a unique boundary incompressible restriction in the core, which is unique up to isotopy.
\end{proof}

The simplest illustration of the polygonal disc phenomenon arises in the basic case of a \textit{torus}$\times I$. An embedded annulus vertical to the product structure can be recharacterised as an embedded, irreducible quadrilateral disc for any embedded \gi{} non-orientable surface. Such a disc allows M\"obius bands to slide between collars of the inner and outer boundaries. Therefore, there is a choice of assigning the band to one or the other boundary collar, whereby different boundary slopes result in the core. See Figures~\ref{fig:torusxI1}, \ref{fig:torusxI2} and \ref{fig:torusxI3} for a comparison of the $(2,1)$ M\"obius band arranged in a collar of the outer boundary (left column) and a collar of the inner boundary (right column).

\begin{figure}[htbp] 
   \centering
   \includegraphics[width=3in]{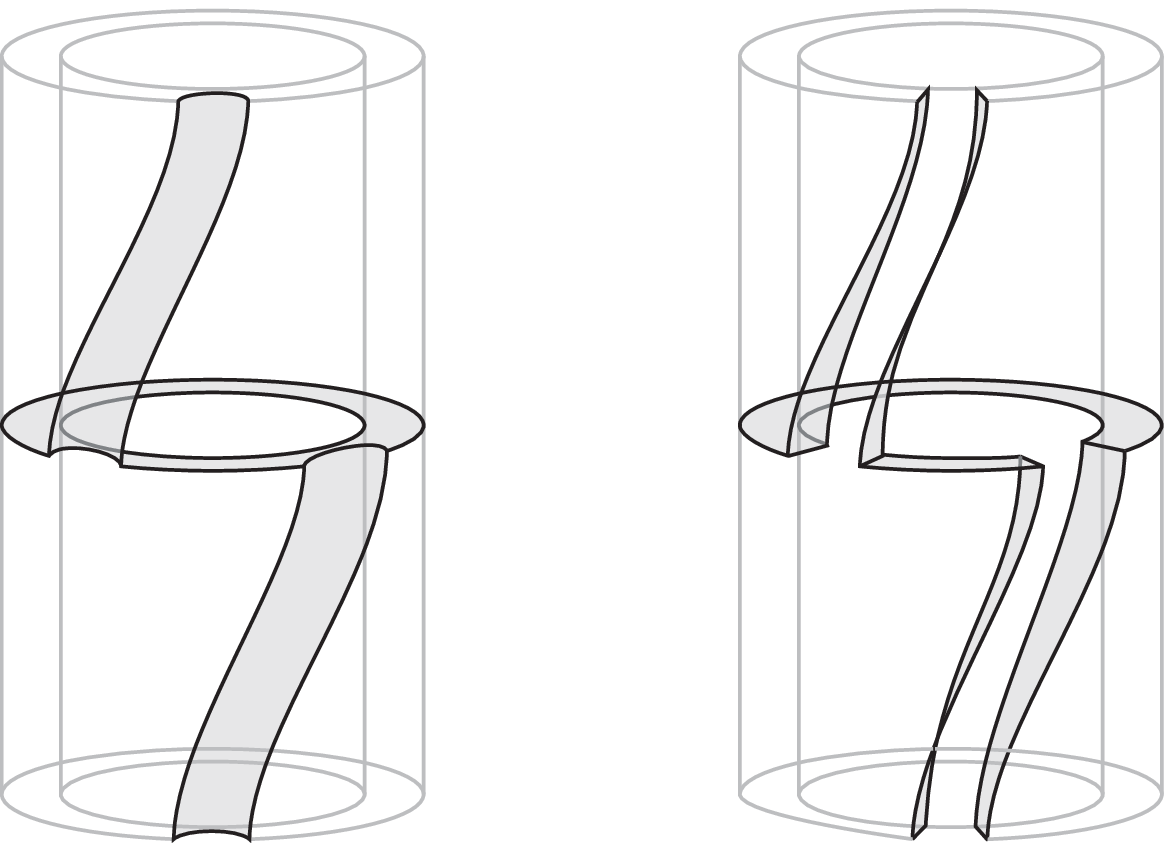} 
   \caption{Collar of $\partial_{+}N$}
   \label{fig:torusxI1}
   \vspace{8mm}
   \includegraphics[width=3in]{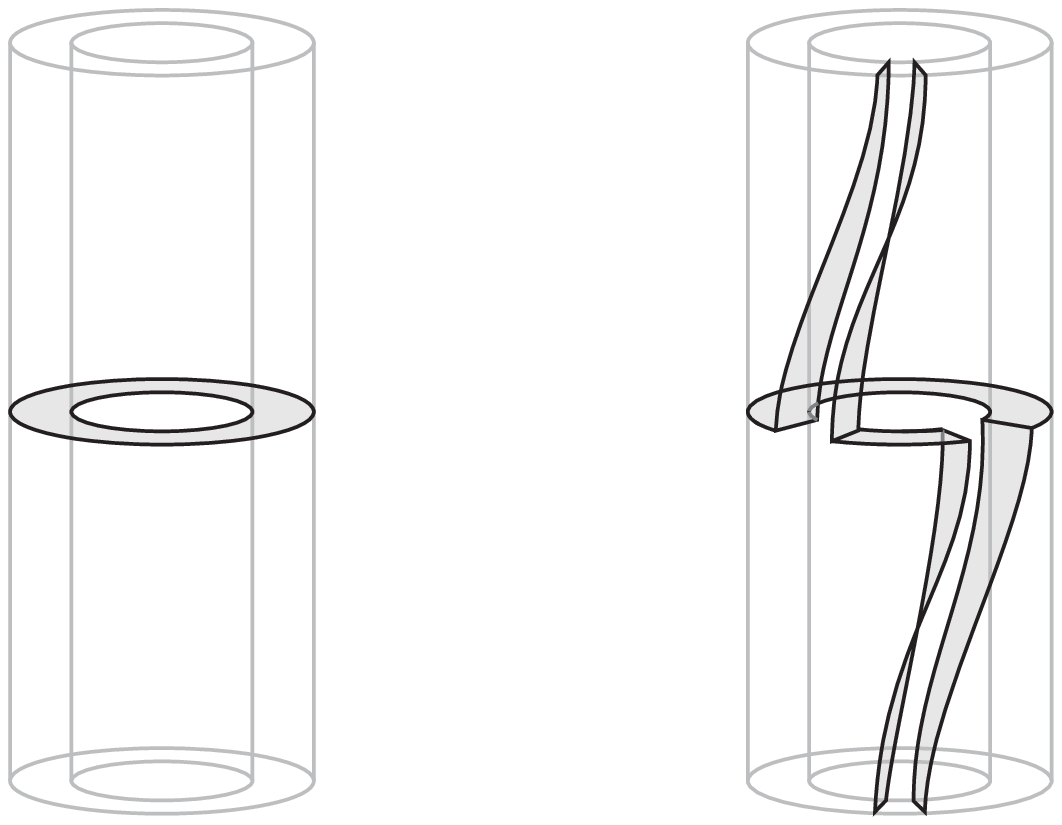} 
   \caption{Non-isotopic boundary incompressible restrictions in the core}
   \label{fig:torusxI2}
\vspace{8mm}
   \centering
   \includegraphics[width=3in]{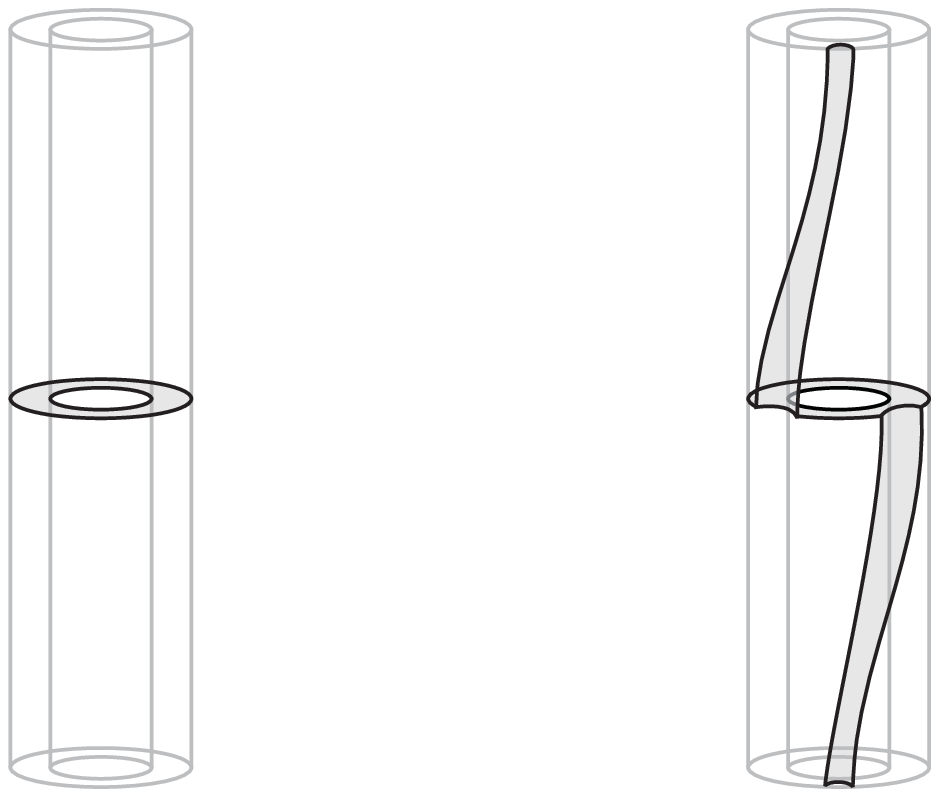} 
   \caption{Collar of $\partial_{-}N$}
   \label{fig:torusxI3}
\end{figure}

As per Proposition~\ref{prop:boundary_incomp}, this non-uniqueness of boundary incompressible restrictions is precluded by the absence of embedded, irreducible polygonal discs. For the exceptional cases where such discs are present, it is expected that there are combinatorial conditions on a polygonal disc in order for it to produce an isotopy that gives different boundary incompressible restrictions; it does not seem likely that such discs will produce different restrictions generically. This is the basis of ongoing work.
\vspace{\baselineskip}

Given the detection of embedded polygonal discs for any given link space is a non-trivial problem, consider a simple class of examples with no such discs:
\vspace{1.25\baselineskip}

\textbf{Examples of knot spaces with no irreducible embedded polygonal discs}

Consider a once punctured torus bundle over the circle $M$ with monodromy $\varphi$. We can classify
$\varphi$ up to conjugacy by its action on the first homology of the fibre, hence a matrix $A$ in $SL(2,\mathbb{Z})$. Denote this matrix by

$$A=\left(\begin{array}{cc}a & b \\c & d\end{array}\right)$$

The aim is to show that any embedded irreducible polygonal disc $d$ in $M$ for the fibre $K$ must be quadrilateral, and to relate existence of such a disc to properties of the matrix $A$. In particular, a large class of examples will be constructed where there are no such discs. 

Suppose first that such a disc $d$ is not quadrilateral. Choose any arc $\lambda$ of $d \cap K$ and form a new quadrilateral disc of the form $d^{\prime}=\lambda \times [0,1]$ using the bundle structure of $M$ as ${(K \times [0,1])} /\varphi$. In general, this disc will have isolated intersections since the embeddings $\lambda = \lambda \times \{0\}$ and $\gamma = \lambda \times \{1\}$ in $K$ will have isolated intersections. Note that $\gamma = \varphi(\lambda)$.

We can arrange that $d \cap d^{\prime}$ consists of arcs with endpoints on $K$, by removing any loops of intersection and choosing the boundary arcs of the discs on $\partial M$ to be disjoint. We also can suppose that the boundary arc $\lambda$ of $d^{\prime}$ is pushed off $d$ so that all the arcs of $d \cap d^{\prime}$ have ends on $\gamma$. Choose an innermost arc $\mu \in d \cap d^{\prime}$ that bounds a bigon in $d^{\prime}$. Cutting $d$ along this bigon shows that either $d$ is reducible or we can isotope $d^{\prime}$ to remove $\mu$. This leads to a contradiction, unless $d$ was already a quadrilateral disc parallel to $d^{\prime}$.

A properly embedded essential arc $\lambda$ on $K$ gives rise to a quadrilateral disc $d^{\prime}$, as above, which is embedded, exactly when the arcs $\lambda, \varphi(\lambda)=\gamma$ are disjoint or meet in one point. If we consider $\lambda$ as an element $(x,y)$ of $H_1(K)= \mathbb{Z} \times \mathbb{Z}$ by joining its ends along $\partial K$, then $\gamma$ becomes the element $(ax+by,cx+dy)$ in $H_1(K)$. Computing intersection number of these classes, we see that the condition required for an embedded quadrilateral disc is that $$(ax+by)y-(cx+dy)x =0,1,-1$$

It is now straightforward to find many matrices for which no solution exists for this equation. For example, suppose $a,d$ are odd and $b,c$ are even, ie. $A \equiv I$ mod $2$. Then the left side of the equation is even, so the only possible solution is the right side is $0$. But as $x,y$ and $ax+by, cx+dy$ are relatively prime pairs, the only solution is $cx+dy =ny, ax+by =nx$ for some integer $n$. But the only integer eigenvalues for $A$ occur when  $A$ has trace $\pm 2$. We conclude if trace $A \ne 2$ and $A \equiv I$ mod $2$, then there are no embedded polygonal discs.
\vspace{1.25\baselineskip}

Having thus established the set of link spaces where no embedded, irreducible polygonal discs are present is non-empty, Proposition~\ref{prop:boundary_incomp} determines that in the core of such spaces, the boundary incompressible restriction of \gi{} surfaces is unique.

\bibliographystyle{abbrv}
\bibliography{oshs}

\begin{thebibliography}{1}

\bibitem{bredon-wood}
G.~E. Bredon and J.~W. Wood.
\newblock Non-orientable surfaces in orientable {$3$}-manifolds.
\newblock {\em Invent. Math.}, 7:83--110, 1969.

\bibitem{rubin78}
J.~H. Rubinstein.
\newblock One-sided {H}eegaard splittings of {$3$}-manifolds.
\newblock {\em Pacific J. Math.}, 76(1):185--200, 1978.

\bibitem{wald_suff_large}
F.~Waldhausen.
\newblock On irreducible {$3$}-manifolds which are sufficiently large.
\newblock {\em Ann. of Math. (2)}, 87:56--88, 1968.

\end{thebibliography}

\textit{Department of Mathematics\\
Oklahoma State University\\
Stillwater, Oklahoma 74078\\
United States}

Email: \texttt{bartolini@math.okstate.edu}

\end{document}